\documentclass[a4paper,10pt]{article}

\usepackage{a4wide,geometry}
\usepackage[utf8]{inputenc}
\usepackage[T1]{fontenc}
\usepackage[english]{babel}

\usepackage{amsmath,amsthm,amssymb}

\newtheorem{theorem}{Theorem}[section]

\newtheorem{lemma}[theorem]{Lemma}

\newtheorem{remark}[theorem]{Remark}

\numberwithin{equation}{section}

\input mydef.sty

\begin{document}

\title{Asymptotic optimality of the edge finite element approximation
of the time-harmonic Maxwell's equations}

\author{
T. Chaumont-Frelet\thanks
{
Inria Univ. Lille and Laboratoire Paul Painlev\'e,
59655 Villeneuve-d'Ascq, France
} {}
and
A. Ern\thanks
{
CERMICS, Ecole des Ponts, 77455 Marne-la-Vall\'ee Cedex 2, 
France and Inria Paris, 75589 Paris, France
}
}

\date{}

\maketitle

\begin{abstract} 
We analyze the conforming approximation of the time-harmonic Maxwell's equations using
N\'ed\'elec (edge) finite elements.
We prove that the approximation is asymptotically optimal, i.e., the approximation 
error in the energy norm is bounded by the best-approximation error 
times a constant that tends to one as the mesh is refined and/or the polynomial
degree is increased. Moreover, under the same conditions on the mesh
and/or the polynomial degree, we establish discrete inf-sup stability with
a constant that corresponds to the continuous constant up to a factor
of two at most. Our proofs apply under minimal regularity assumptions on the
exact solution, so that general domains, material coefficients, and right-hand
sides are allowed.
\end{abstract}

\section{Introduction}

This work analyzes the conforming approximation of the time-harmonic
Maxwell's equations using N\'ed\'elec (edge) finite elements. In this
section, we present the model problem, outline the main challenges in
its finite element approximation, and discuss the present contributions
in view of the existing literature. 

\subsection{Setting}

Let $\Dom\subset \Real^d$, $d=3$, be an open, bounded, Lipschitz polyhedron 
with boundary $\front$ and outward unit normal $\bn_\Dom$. 
We do not make any simplifying assumption on the topology of $\Dom$.
We use boldface fonts for vectors, vector fields, and
functional spaces composed of such fields. More details
on the notation are given in Section~\ref{sec:continuous}. 

Given a positive real number $\omega>0$ representing a frequency and a source
term $\bJ: \Dom \to \mathbb R^3$, and focusing for simplicity on homogeneous
Dirichlet boundary conditions (a.k.a.~perfect electric conductor boundary conditions),
the model problem consists in finding $\bE: \Dom \to \mathbb R^3$ such that
\begin{subequations}
\label{eq_maxwell_strong}
\begin{alignat}{2}
-\omega^2 \eps \bE+\ROT(\bmu^{-1}\ROT \bE) &= \bJ &\quad&\text{in $\Dom$},
\\
\bE \CROSS \bn_\Dom &= \bzero &\quad&\text{on $\partial \Dom$},
\end{alignat}
\end{subequations}
where $\eps$ represents the electric permittivity of the
materials contained in $\Dom$ and $\bmu$ their magnetic permeability. 
Both material properties can vary in $\Dom$ and take symmetric
positive-definite values with eigenvalues uniformly bounded from above and
from below away from zero. 
We assume that $\omega$ is not a resonant frequency,
so that \eqref{eq_maxwell_strong} is uniquely solvable in $\Hrotz$
for every $\bJ$ in the topological dual space $\Hrotz'$. 
The time-harmonic
Maxwell's equations \eqref{eq_maxwell_strong} are one of the central models
of electrodynamics. Therefore, efficient discretizations are
a cornerstone for the computational modelling of electromagnetic 
wave propagation \cite{Hiptmair_Acta_Num_2002,Monk_book_2003}.

\subsection{Main challenges when approximating \eqref{eq_maxwell_strong}}

To highlight the main challenges associated with the finite element approximation of
\eqref{eq_maxwell_strong}, let us first briefly discuss 
the Helmholtz problem that arises when considering
polarized electromagnetic waves. In this case, a two-dimensional domain
$\widehat \Dom \subset \mathbb R^2$ is considered, and the component of the electric
field normal to $\widehat \Dom$, $\widehat E: \widehat \Dom \to \mathbb R$, satisfies
\begin{subequations}
\label{eq_helmholtz_strong}
\begin{alignat}{2}
-\omega^2 \widehat\epsilon \widehat E-\DIV(\widehat\bmu^{-1}\GRAD \widehat E) &= \widehat J &\quad&\text{in $\widehat \Dom$}, \label{eq:EDP_Helmholtz} \\
\widehat E &= 0&\quad&\text{on $\partial \widehat \Dom$},
\end{alignat}
\end{subequations}
where $\widehat J: \widehat \Dom \to \mathbb R$ is the component of 
$\bJ$ normal to $\widehat \Dom$, $\widehat\epsilon$ the normal-normal
component of $\eps$, and $\widehat\bmu$ the transpose-adjugate of the
tangent-tangent components of $\bmu$. The Helmholtz problem also arises
in other contexts, such as (three-dimensional) acoustic wave propagation.

The Laplace operator in \eqref{eq_helmholtz_strong} is coercive over $H^1_0(\widehat \Dom)$,
and the compact embedding $H^1_0(\widehat \Dom) \hookrightarrow 
L^2(\widehat \Dom)$ ensures that
the negative term in~\eqref{eq:EDP_Helmholtz}
is a compact perturbation. As a result, at the continuous level,
\eqref{eq_helmholtz_strong} falls into the framework of the Fredholm alternative, and
the Helmholtz problem is well-posed in $H^1_0(\widehat \Dom)$ 
if and only if $\omega$ is not a resonant frequency.
Compactness also has direct implications at the discrete
level. The most common manifestation is probably 
the celebrated Aubin--Nitsche duality argument.
Specifically, setting $\omega = 0$ and considering conforming 
Lagrange finite elements for
simplicity, it is well-known that the finite element approximation, 
$\widehat E_h$, converges to $\widehat E$
faster in the $L^2(\widehat \Dom)$-norm than in the 
$H^1_0(\widehat \Dom)$-norm.

When $\omega > 0$, this observation can be leveraged into a 
technique often called Schatz's argument \cite{Schatz:74}. 
The key idea is that 
the negative $L^2(\widehat \Dom)$-term in \eqref{eq_helmholtz_strong}
becomes negligible on sufficiently fine meshes, 
leaving only a coercive term and thus enabling the derivation of a 
C\'ea-like lemma. 
More specifically, considering the Lagrange finite element space
$\widehat V_h \subset H^1_0(\widehat \Dom)$
with mesh size $h$ and polynomial degree $k\ge1$,
one can show that if the mesh is sufficiently refined
and/or the polynomial degree is sufficiently increased,
the finite element approximation, $\widehat E_h$, is 
uniquely defined and satisfies the following error bound:
\begin{equation}
\label{eq_optimality_helmholtz}
(1-\gamma) \tnorm{\widehat E-\widehat E_h}
\leq
\min_{\widehat v_h \in \widehat V_h} \tnorm{\widehat E-\widehat v_h},
\end{equation}
where the approximation factor $\gamma$ satisfies
$\lim_{{h/k} \to 0} \gamma = 0$, and
\begin{equation}
\tnorm{w}^2 \eqq \omega^2\|\epsilon^{\frac12}w\|_{\Ldeux}^2 + \|\bmu^{-\frac12}\GRAD w\|_{\Ldeuxd}^2,
\end{equation}
is the natural energy norm for~\eqref{eq_helmholtz_strong}. Crucially, \eqref{eq_optimality_helmholtz} implies
that the finite element approximation is \emph{asymptotically optimal}.
The original argument of Schatz in \cite{Schatz:74} for conforming finite elements
requires some extra smoothness on the solution, and it has been extended in a number of ways.
Of particular importance is the seminal work \cite{melenk_sauter_2010a}
which tracks the dependence of $\gamma$ on the frequency $\omega$, the mesh size
$h$ and the polynomial degree $k$. This has later been extended to non-smooth 
domains and varying
coefficients \cite{chaumontfrelet_nicaise_2020a,LaSpW:22,melenk_sauter_2011a}. 

The challenges encountered in the Helmholtz problem~\eqref{eq_helmholtz_strong}
are also present in the time-harmonic Maxwell's equations \eqref{eq_maxwell_strong},
with one key additional difficulty: the lack of compactness of the
embedding $\Hrotz \hookrightarrow \Ldeuxd$. This is remedied by working in
the subspace
$\Hrotz \cap \Hdiveps$, with $\Hdiveps := \{ \bv \in \bL^2(D); \DIV(\eps\bv) \in L^2(D) \}$,
which compactly embeds into $\Ldeuxd$ \cite{Birman_Solomyak_1987,Costabel:90,Weber:80}.
Therefore, a crucial ingredient in the analysis of any finite element
approximation to \eqref{eq_maxwell_strong} is to derive some suitable
control on the divergence of the discrete solution.
This is discussed in the context of Lagrange finite elements, e.g.,
in \cite{Costabel:91,CosDa:02,BoGuL:16,BuCiJ:09}. However, 
a somewhat more popular approach
to approximate the time-harmonic Maxwell's equations in a conforming setting
hinges on N\'ed\'elec finite
elements \cite{Nedel:80,Nedel:86}.
We now discuss our main contributions to this topic.

\subsection{Main contributions}

The N\'ed\'elec finite element approximation to \eqref{eq_maxwell_strong} 
is $\Hrotz$-conforming,
but only weakly $\Hdiveps$-conforming. The lack of $\Hdiveps$-conformity must be taken 
into account in the error analysis. Early contributions on the topic
are based on the concept of collective compactness.
A seminal work in this regard is \cite{Kikuchi:89} for lowest-order N\'ed\'elec elements,
and extensions to higher-order elements have been carried out 
\cite{BCDDH:11,MonDe:01}. We also refer the reader to \cite{Buffa:05,CaFeR:00}, and to
\cite[Section 7.3]{Monk_book_2003} for an overview.

Later on, duality proofs in the spirit of Schatz were proposed. 
For Maxwell's equations,
the Aubin--Nitsche trick cannot be applied directly, 
so that an intermediate step involving
a commuting interpolation operator is added to the proof. 
This argument seems to date back to
the work of Girault \cite{Girault:88} using canonical
interpolation operators, and has been used for Maxwell's equations in
\cite{Monk:92,ZhSWX:09}. As pointed out in
\cite[Remark 3.1]{Girault:88}, one drawback of using the canonical
interpolation operators is the extra regularity requirement on the exact solution. 
This limitation was lifted following
the development of commuting quasi-interpolation operators working under minimal regularity
assumptions (see \cite{Schoberl:01,ArnFW:06,Christiansen:07,ChrWi:08} and \cite[Chap.~22-23]{EG_volI}). The application
to Maxwell's equations can be found in
\cite{chaumontfrelet_nicaise_pardo_2018a,ErnGu:18},
see also \cite[Chap.~44]{EG_volII}.

The current state-of-the-art using the above techniques 
shows that the N\'ed\'elec finite element 
approximation is uniquely defined
for fine enough meshes, and allows for error estimates of the form
\begin{equation}
\label{eq_quasi_optimality_maxwell}
(1-\gamma) \tnorm{\bE-\bE_h}
\leq
c \min_{\bv_h \in \bV_h\upc} \tnorm{\bE-\bv_h},
\end{equation}
where $\bV_h\upc$ is the N\'ed\'elec finite element space with mesh size $h$ and
polynomial degree $k\ge0$, the approximation factor $\gamma$ again satisfies
$\lim_{h/(k+1) \to 0} \gamma = 0$, and the natural energy norm is
\begin{equation} \label{eq:energy_norm}
\tnorm{\bw}^2
\eqq
\omega^2\|\eps^{\frac12}\bw\|_{\Ldeuxd}^2 + \|\bmu^{-\frac12}\ROT \bw\|_{\Ldeuxd}^2.
\end{equation}
Unfortunately, $c$ is a generic constant that depends on the shape-regularity
of the mesh.

The key contribution of this work is to show that \eqref{eq_quasi_optimality_maxwell}
actually holds true with $c = 1$. In other words, N\'ed\'elec finite element discretizations 
of the time-harmonic Maxwell's equations are
\emph{asymptotically optimal}. Moreover, we establish
discrete inf-sup stability with a constant that corresponds, as the mesh is refined,
to the continuous constant up to a factor of two at most.
These results are, to the best of our knowledge, a novel contribution
to the literature. Our proofs are valid irrespective of
the topology of the domain $\Omega$ and apply under a minimal regularity assumption 
on the exact solution, so that general
domains, material coefficients and right-hand sides are allowed. 
In addition, the dependence of
$\gamma$ on key parameters can be traced following
\cite{ChaVe:22,melenk_sauter_2021}.

\subsection{Outline}

The paper is organized as follows. In Section~\ref{sec:continuous}, we briefly
present the continuous setting, and in Section~\ref{sec:disc_setting}, we 
do the same for the discrete setting. 
In Section~\ref{sec:conforming}, we present the error and stability analysis.
The main results in this section are Theorem~\ref{th:est_err_c} and
Theorem~\ref{th:inf_sup_c}.
Finally, in Section~\ref{sec:bnd_app_fac}, we establish bounds on the 
approximation and divergence conformity factors introduced in the analysis,
proving that these quantities tend to zero as the mesh is refined.
Incidentally, we notice that the results established in Section~\ref{sec:conforming}
hold more generally when working on a generic subspace of $\Hrotz$ which is not necessarily
constructed using finite elements. The finite element structure is used in
Section~\ref{sec:bnd_app_fac}.

\section{Continuous setting}
\label{sec:continuous}

In this section, we briefly recall the functional setting for the time-harmonic
Maxwell's equations, formulate the model problem and examine its inf-sup stability.

\subsection{Functional spaces}

We use standard notation for Lebesgue and Sobolev spaces. 
To alleviate the notation,
the inner product and associated norm in the spaces $\Ldeux$ and $\Ldeuxd$
are denoted by $(\SCAL,\SCAL)$ and $\|\SCAL\|$, respectively. The 
material properties $\eps$ and $\bnu\eqq \bmu^{-1}$ are measurable functions
that take symmetric positive-definite values in
$\Dom$ with eigenvalues uniformly bounded from above and from below away from zero.
It is convenient to introduce  
the inner product and associated norm weighted by either $\eps$ or $\bnu$,
leading to the notation $(\SCAL,\SCAL)_\eps$, $\|\SCAL\|_\eps$, $(\SCAL,\SCAL)_{\bnu}$
and $\|\SCAL\|_{\bnu}$. Whenever no confusion can arise, we use
the symbol $^\perp$ to denote orthogonality with respect to the inner product
$(\SCAL,\SCAL)_\eps$. Moreover, all the projection operators denoted using the 
symbol $\bPi$ are meant to be
orthogonal with respect to this inner product; we say that the projections are 
$\bL^2_\eps$-orthogonal.

We consider the Hilbert Sobolev spaces
\begin{subequations} \label{eq:Hrot_spaces} \begin{align}
\Hrot & \eqq \{\bv \in \Ldeuxd  \st \ROT\bv\in \Ldeuxd\},\\
\Hrotrotz &\eqq \{\bv \in \Hrot \st \ROT\bv=\bzero\},   \\
\Hrotz &\eqq \{\bv \in \Hrot \st \gamma\upc_\front(\bv)=\bzero\}, \\
\Hrotzrotz &\eqq \{\bv \in \Hrotz \st \ROTZ\bv=\bzero\},          
\end{align} \end{subequations}
where the tangential trace operator
$\gamma\upc_{\front}:\Hrot\rightarrow \bH^{-\frac12}(\front)$ is
the extension by density of the tangent trace operator such
that $\gamma\upc_\front(\bv)=\bv|_{\front}\CROSS \bn_\Dom$ for smooth fields.
The subscript ${}_0$ indicates the curl operator acting on fields
respecting homogeneous Dirichlet conditions. 
Notice that $\ROT$ and $\ROTZ$ are adjoint to each other, i.e.,
$(\ROTZ\bv,\bw)=(\bv,\ROT\bw)$ for all
$(\bv,\bw)\in \Hrotz\times\Hrot$.

We consider the subspace
\begin{equation} 
\bX\upc_0\eqq \Hrotz\cap \Hrotzrotz^\perp, 
\end{equation}
and we introduce the $\bL^2_\eps$-orthogonal projection
\begin{equation}
\bPi\upc_0: \Ldeuxd \to \Hrotzrotz.
\end{equation}
Since $\GRAD\Hunz \subset \Hrotzrotz$, any field $\bxi \in \bX\upc_0$ 
is such that $\DIV(\eps\bxi)=0$ in $\Dom$. Hence,
$\bX\upc_0$ compactly embeds into $\Ldeuxd$ \cite{Weber:80}.

\begin{remark}[Topology of $\Dom$]
We have $\Hrotzrotz^\perp\subset \{\bv\in \Ldeuxd, \; \DIV(\eps\bv)=0\}$ 
with equality if only if $\front$ is connected (see, e.g., \cite{AmBDG:98}).
\end{remark}

\subsection{Model problem}

We focus for simplicity on homogeneous Dirichlet boundary conditions and consider the
functional space
\begin{equation}
\bV_0\upc \eqq \Hrotz.
\end{equation}
Given a positive real number $\omega>0$ and a source term 
$\bJ \in (\bV_0\upc)'$ (the topological dual space of $\bV_0\upc$),
the model problem amounts to finding $\bE\in\bV_0\upc$ such that
\begin{equation} \label{eq:weak}
b(\bE,\bw) = \langle\bJ,\bw\rangle \qquad \forall \bw \in \bV_0\upc,
\end{equation}
with the bilinear form defined on $\bV_0\upc\times\bV_0\upc$ such that
\begin{equation}
b(\bv,\bw) \eqq -\omega^2(\bv,\bw)_\eps + (\ROTZ\bv,\ROTZ\bw)_{\bnu},
\end{equation} 
and where the brackets on the right-hand side of~\eqref{eq:weak} denote the duality
product between $(\bV_0\upc)'$ and $\bV_0\upc$.
In what follows, we assume that $\omega^2$ is not an eigenvalue of the 
$\ROT(\bnu\ROTZ {\cdot})$ operator in $\Dom$. As a result, the model problem~\eqref{eq:weak} 
is well-posed.
We equip the space $\Hrot$ and its subspaces defined in~\eqref{eq:Hrot_spaces} 
with the energy norm defined in~\eqref{eq:energy_norm},
and observe that the bilinear form $b$ satisfies $|b(\bv,\bw)|\le \tnorm{\bv}\tnorm{\bw}$.

\subsection{Inf-sup stability}

For all $\bg \in \Ldeuxd$, let $\bv_{\bg}\in \bV_0\upc$ denote the unique
solution to~\eqref{eq:weak} with right-hand side $(\bg,\bw)_\eps$,
i.e., $b(\bv_\bg,\bw)=(\bg,\bw)_\eps$ for all $\bw\in \bV_0\upc$.
We introduce the (nondimensional) stability constant
\begin{equation} \label{eq:def_bst}
\bst \eqq
\sup_{\substack{\bg \in \Hrotzrotz^\perp \\ \|\bg\|_\eps = 1}} \omega \tnorm{\bv_{\bg}}.
\end{equation}

\begin{lemma}[Inf-sup stability] \label{lem:infsup}
The following holds:
\begin{equation} \label{eq:infsup_exact}
\frac{1}{1+2\bst} \le \inf_{\substack{\bv \in \bV_0\upc\\ \tnorm{\bv}=1}} \sup_{\substack{\bw\in \bV_0\upc \\ \tnorm{\bw}=1}} |b(\bv,\bw)| \le \frac{1}{\bst}.
\end{equation}
\end{lemma}

\begin{proof}
(1) Lower bound. Let $\bv \in \bV_0\upc$ and let us set $\bv=\bv_0+\bv_\Pi$ with
$\bv_0\eqq (I-\bPi\upc_0)(\bv)$ and $\bv_\Pi\eqq \bPi\upc_0(\bv)$. 
Let $\bxi_0 \in \bV_0\upc$ be the adjoint solution such that 
$b(\bw,\bxi_0)=\omega^2 (\bw,\bv_0)_\eps$ for all $\bw\in\bV_0\upc$. Since 
$\bv_0\in \Hrotzrotz^\perp$ by construction, we have
\begin{equation} \label{eq:pty_bxi}
\tnorm{\bxi_0} \le \bst \omega \|\bv_0\|_\eps \le \bst \tnorm{\bv_0},
\end{equation}
owing to the symmetry of $b$ and the definition of $\bst$ for the first bound,
and the definition of the $\tnorm{\SCAL}$-norm for the second bound.
Moreover, taking the test function $\bw\eqq \bv\in\bV_0\upc$ in the adjoint problem,
we infer that
\[
b(\bv,\bxi_0) = \omega^2(\bv,\bv_0)_\eps=\omega^2\|\bv_0\|_\eps^2,
\]
since $\bv_0$ and $\bv_\Pi$ are $\bL^2_\eps$-orthogonal.
In addition, invoking the symmetry of $b$ and since $\bv_\Pi$ is curl-free, we have
\begin{align*}
b(\bv,\bv_0-\bv_\Pi) &=b(\bv_0+\bv_\Pi,\bv_0-\bv_\Pi) 
=b(\bv_0,\bv_0)-b(\bv_\Pi,\bv_\Pi) \\
&=\tnorm{\bv_0}^2 -2\omega^2\|\bv_0\|_\eps^2 + \omega^2\|\bv_\Pi\|_\eps^2
=\tnorm{\bv}^2-2\omega^2\|\bv_0\|_\eps^2.
\end{align*}
Combining the above two identities proves that
\[
b(\bv,\bv_0+2\bxi_0-\bv_\Pi) = \tnorm{\bv}^2.
\]
Finally, owing to~\eqref{eq:pty_bxi}, we have
\begin{align*}
\tnorm{\bv_0+2\bxi_0-\bv_\Pi}^2 &= \tnorm{\bv_0+2\bxi_0}^2 + \omega^2\|\bv_\Pi\|_\eps^2 \\
&\le (1+2\bst)^2 \tnorm{\bv_0}^2 + \omega^2\|\bv_\Pi\|_\eps^2 \le (1+2\bst)^2 \tnorm{\bv}^2,
\end{align*}
since $\tnorm{\bv_0}^2 + \omega^2\|\bv_\Pi\|_\eps^2 = \tnorm{\bv}^2$.
This proves the lower bound.  

(2) Upper bound. For all $\bphi\in (\bV_0\upc)'$, let $\bv_\bphi$ denote the unique
solution to \eqref{eq:weak} with right-hand side $\langle \bphi,\bw\rangle$. 
Let $\alpha$ denote the inf-sup constant in~\eqref{eq:infsup_exact}. 
Then $\alpha>0$ since \eqref{eq:weak} is well posed, and we have
$\alpha^{-1} = \sup_{\bphi\in (\bV_0\upc)'}\frac{\tnorm{\bv_\bphi}}{\|\bphi\|_{(\bV_0\upc)'}}$
(see, e.g., \cite[Lem.~C.51]{EG_volII}).
We consider the linear forms $\bphi_\bg\in(\bV_0\upc)'$ such 
that $\langle \bphi_\bg,\bw\rangle = (\bg,\bw)_\eps$ for all $\bw\in\bV_0\upc$
for some function $\bg\in\Hrotzrotz^\perp$. Owing to the Cauchy--Schwarz inequality and the definition of the $\tnorm{\SCAL}$-norm, we have 
\[
\|\bphi_\bg\|_{(\bV_0\upc)'} = \sup_{\bw\in \bV_0\upc} \frac{|\langle \bphi_\bg,\bw\rangle|}{\tnorm{\bw}}
\le \sup_{\bw\in \bV_0\upc} \frac{\|\bg\|_\eps\|\bw\|_\eps}{\tnorm{\bw}}\le \omega^{-1}\|\bg\|_\eps.
\] 
Restricting the supremum defining $\alpha^{-1}$ to the above linear forms,
we infer that
\[
\alpha^{-1} \ge \sup_{\bg\in \Hrotzrotz^\perp}\frac{\tnorm{\bv_\bg}}{\omega^{-1}\|\bg\|_\eps} = \bst.
\]
This proves the upper bound.
\end{proof}

\begin{remark}[Inf-sup condition] \label{rem:exact_infsup}
Since the stability constant $\bst$ is expected to be large (it
scales as $\omega$ times the reciprocal of the distance of $\omega$ to the spectrum 
of the $\ROT(\bnu\ROTZ{\cdot})$ operator \cite{ChaVe:22}), \eqref{eq:infsup_exact}
means that, up to at most a factor of two, the inf-sup constant of the bilinear form $b$
can be estimated as $(1+2\bst)^{-1}$.
\end{remark}

\section{Discrete setting}
\label{sec:disc_setting}

In this section, we present the
setting to formulate the discrete problem,
and we introduce two (nondimensional) quantities to be used in the error
and stability analysis presented in the next section.

\subsection{Mesh and finite element spaces}

Let $\calT_h$ be an affine simplicial mesh covering $\Dom$ exactly.
A generic mesh cell is denoted $K$, its diameter $h_K$ and its outward unit normal $\bn_K$.
Let $k\ge0$ be the polynomial degree. Let $\polP_{k,d}$ be the space
composed of $d$-variate polynomials of total degree at most $k$ and set
$\bpolP_{k,d}\eqq [\polP_{k,d}]^d$.  Let $\bpolN_{k,d}$ be the space composed of
the $d$-variate N\'ed\'elec polynomials of order $k$ of the first kind
(recall that $\bpolP_{k,d}\subsetneq \bpolN_{k,d} \subsetneq \bpolP_{k+1,d}$).
We consider the discrete subspace:
\begin{equation} 
\bV_{h0}\upc \eqq \bset \bv_h\in
\bV_0\upc \st \bv_h|_K\in \bpolN_{k,d},\, \forall K\in\calT_h\eset. \label{eq:def_Vhc}
\end{equation}
Moreover, we let
\begin{equation}
\bV\upc_{h0}(\ceqz) \eqq \bset \bv_h \in \bV_{h0}\upc \st  \ROTZ \bv_h = \bzero \eset.
\end{equation}
The $\bL^2_\eps$-orthogonal projection 
\begin{equation}
\bPi\upc_{h0} : \Ldeuxd \to \bV\upc_{h0}(\ceqz)
\end{equation}
plays a key role in what follows. In particular, we introduce the subspace
\begin{equation} \label{eq:disc_involution}
\bX\upc_{h0}\eqq \bV_{h0}\upc \cap\bV\upc_{h0}(\ceqz)^\perp,
\end{equation}
which is composed of fields $\bv_h$ such that $\bPi\upc_{h0}(\bv_h)=\bzero$.
Loosely speaking, discrete fields in $\bX\upc_{h0}$ are discretely divergence-free
(and satisfy a finite number of additional constraints when $\front$ is not
connected).

\subsection{Discrete problem}

The discrete problem reads as follows: Find $\bE_h\in \bV_{h0}\upc$ such that
\begin{equation} \label{eq:disc_pb_c}
b(\bE_h,\bw_h)=\langle \bJ,\bw_h\rangle \qquad \forall \bw_h\in \bV_{h0}\upc.
\end{equation}
The well-posedness of~\eqref{eq:disc_pb_c} is established in Section~\ref{sec:conforming}.

\subsection{Approximation and divergence conformity factors}
\label{sec:def_factors_c}

In this section, we define two factors, the approximation factor and the
divergence conformity factor. 
Both factors are used in the error and inf-sup stability analysis. We prove in
Section~\ref{sec:bnd_app_fac} that both factors tend to zero
as the mesh size tends to zero or the polynomial degree is increased.

For all $\btheta \in \Hrotzrotz^\perp$, we consider the adjoint problem
consisting of finding $\bzeta_\btheta \in \bV_0\upc$ such that
\begin{equation} \label{eq:adjoint}
b(\bw,\bzeta_\btheta) = (\bw,\btheta)_\eps \qquad \forall \bw \in \bV_0\upc.
\end{equation}
Taking any test function $\bw \in \Hrotzrotz\subset \bV_0\upc$ shows that 
\begin{equation}
\omega^2(\bw,\bzeta_\btheta)_\eps=b(\bw,\bzeta_\btheta)=\omega^2(\bw,\btheta)_\eps=0,
\end{equation} 
where the first equality follows from $\ROTZ\bw=\bzero$, the second from the
definition of the adjoint solution, and the third from the assumption 
$\btheta \in \Hrotzrotz^\perp$.
Since $\bw$ is arbitrary in $\Hrotzrotz$, this proves that
$\bzeta_\btheta \in \Hrotzrotz^\perp$. Thus, $\bzeta_\btheta \in \bX\upc_0$,
and owing to the compact embedding $\bX\upc_0\hookrightarrow \Ldeuxd$, it is
reasonable to expect that the (nondimensional) approximation factor
\begin{equation}  
\gpc\eqq \sup_{\substack{\btheta \in \Hrotzrotz^\perp\\ \|\btheta\|_\eps=1}} \min_{\bv_h\upc\in \bV\upc_{h0}}
\omega \tnorm{\bzeta_\btheta-\bv_h\upc}, \label{eq:def_gamma_c}
\end{equation}
exhibits some decay rate as the mesh size tends to zero
and/or the polynomial degree is increased.

The second useful quantity is the (nondimensional) divergence conformity factor
\begin{equation}
\label{eq_gamma_bX}
\gdivc \eqq 
\sup_{\substack{\bv_h \in \bX\upc_{h0} \\ \|\ROTZ \bv_h\|_\bnu = 1}}
\omega \|\Pi\upc_0(\bv_h)\|_\eps.
\end{equation}
Loosely speaking, the adopted terminology reminds us that $\gdivc$ essentially
measures how much a discrete field that is discretely divergence-free is not
pointwise divergence-free. It is again reasonable to expect, under the same
conditions as above, that $\gdivc$ tends to zero. 


\section{Error and stability analysis}
\label{sec:conforming}

In this section, we analyze the conforming approximation~\eqref{eq:disc_pb_c} 
of the model problem~\eqref{eq:weak}.
As usual in duality arguments, we first establish an error estimate under a 
smallness condition on the mesh size
(this estimate implies the uniqueness, and thus also the existence, 
of the discrete solution under the same
condition on the mesh size), and then establish an inf-sup condition 
again under the same assumptions.

\subsection{Error decomposition and preliminary bounds} 

Let us define the bilinear form $b^+$ on $\bV_0\upc\times \bV_0\upc$ such that
\begin{equation}
b^+(\bv,\bw) \eqq \omega^2(\bv,\bw)_\eps + (\ROTZ\bv,\ROTZ\bw)_\bnu.
\end{equation}
We define the best-approximation operator
$\bestc:\bV_0\upc\rightarrow \bV\upc_{h0}$ as follows: For
all $\bv \in \bV_0\upc$, $\bestc(\bv) \in \bV\upc_{h0}$ is such that
\begin{equation} \label{eq:def_best_c}
b^+(\bv-\bestc(\bv),\bw_h) = 0 \quad \forall \bw_h\in \bV\upc_{h0}.
\end{equation}

\begin{lemma}[Properties of $\bestc$] \label{lem:best}
The best-approximation operator $\bestc$ defined in~\eqref{eq:def_best_c} enjoys the following two properties:
\begin{subequations} \begin{alignat}{2}
&\tnorm{\bestc(\bv)} \le \tnorm{\bv},&\qquad&\forall \bv\in \bV_0\upc, \label{eq:stab_best} \\
&\bestc(\bv) \in \bV\upc_{h0}(\ceqz)^\perp,&\qquad&\forall \bv \in \bV\upc_{h0}(\ceqz)^\perp. \label{eq:useful_pty}
\end{alignat} \end{subequations}
\end{lemma}

\begin{proof}
\eqref{eq:stab_best} follows from the fact that the bilinear form $b^+$ is the inner product
associated with the $\tnorm{\SCAL}$-norm. To prove~\eqref{eq:useful_pty}, take any
$\bw_h \in \bV\upc_{h0}(\ceqz)$ in~\eqref{eq:def_best_c} and observe that
$\omega^2(\bv-\bestc(\bv),\bw_h)_\eps = b^+(\bv-\bestc(\bv),\bw_h) = 0$.
\end{proof}

We define the approximation error and the best-approximation error as follows: 
\begin{equation}
\be\eqq \bE-\bE_h, \qquad \beeta\eqq \bE-\bestc(\bE).
\end{equation}
We consider the error decomposition
$\be = \btheta_0 + \btheta_\Pi$,
with
\begin{equation} \label{eq:def_theta} 
\btheta_0\eqq(I-\bPi\upc_0)(\be)\in \Hrotzrotz^\perp,
\qquad
\btheta_\Pi\eqq\bPi\upc_0(\be) \in \Hrotzrotz.
\end{equation}
The motivation behind this decomposition is that $\omega\|\btheta_0\|_\eps$ can be
bounded by a duality argument, whereas $\omega\|\btheta_\Pi\|_\eps$ represents
a divergence conformity error.

\begin{lemma}[Bound on $\btheta_0$] \label{lem:bnd_thet1_c}
We have
\begin{equation} \label{eq:bnd_thet1_c}
\omega \|\btheta_0\|_\eps \le \gpc\, \tnorm{\be},
\end{equation}
with the approximation factor $\gpc$ defined in~\eqref{eq:def_gamma_c}.
\end{lemma}

\begin{proof}
We consider
the adjoint problem~\eqref{eq:adjoint} with data $\btheta\eqq\btheta_0$. 
Notice that $\btheta\in\Hrotzrotz^\perp$ by construction.
Let $\bzeta_\btheta \in \bX\upc_0$ denote the
unique adjoint solution to this problem. Since $\be \in \bV_0\upc$, we have 
\[
b(\be,\bzeta_\btheta)=(\be,\btheta_0)_\eps.
\]
Using Galerkin's orthogonality together with
$(\btheta_0,\btheta_\Pi)_\eps=0$, and multiplying by $\omega^2$ gives
\[
\omega^2\|\btheta_0\|_\eps^2 = \omega^2(\be,\btheta_0)_\eps = \omega^2b(\be,\bzeta_\btheta-\bv_h) 
\qquad \forall \bv_h\in \bV\upc_{h0}.
\]
Owing to the boundedness of the bilinear form $b$, we infer that
\[
\omega^2\|\btheta_0\|_\eps^2 \le \omega^2 \tnorm{\be} \, \tnorm{\bzeta_\btheta-\bv_h}.
\]
Invoking the approximation factor $\gpc$ gives
\[
\omega^2\|\btheta_0\|_\eps^2\le \tnorm{\be}\, \gpc \omega \|\btheta_0\|_\eps,
\] 
since $\bv_h$ is arbitrary. This proves the claim. 
\end{proof}

\begin{lemma}[Bound on $\btheta_\Pi$] \label{lem:bnd_thet2_c}
We have
\begin{equation} \label{eq:bnd_thet2_c}
(1-\gdivc)\omega \|\btheta_\Pi\|_\eps \le \omega \|\bPi\upc_0(\beeta)\|_\eps 
+ \gdivc\tnorm{\btheta_0}.
\end{equation}
\end{lemma}

\begin{proof}
We observe that $\bPi\upc_{h0}(\btheta_\Pi)=\bPi\upc_{h0}(\bPi\upc_0(\be))
=\bPi\upc_{h0}(\be)=\bzero$ owing to Galerkin's orthogonality. Indeed,
$\omega^2(\be,\bw_h)_\eps=b(\be,\bw_h)=0$ for all $\bw_h\in \bV\upc_{h0}(\ceqz)$.
Hence, we have
\begin{equation} \label{eq:theta_Pi_perp_c}
\btheta_\Pi \in \bV\upc_{h0}(\ceqz)^\perp.
\end{equation}
To bound $\btheta_\Pi$, we write
\[
\|\btheta_\Pi\|_\eps^2 = (\btheta_\Pi,\bestc(\btheta_\Pi))_\eps
+ (\btheta_\Pi,\btheta_\Pi-\bestc(\btheta_\Pi))_\eps \qqe \Theta_1+\Theta_2,
\]
and estimate the two terms on the right-hand side. 
Since $\btheta_\Pi=\bPi\upc_0(\btheta_\Pi)$ and 
$\bPi\upc_0$ is self-adjoint for the inner product $(\SCAL,\SCAL)_\eps$,
we obtain
\begin{align*}
\Theta_1 = (\btheta_\Pi,\bestc(\btheta_\Pi))_\eps &= (\btheta_\Pi,\bPi\upc_0(\bestc(\btheta_\Pi)))_\eps \\
&\le \|\btheta_\Pi\|_\eps \, \|\bPi\upc_0(\bestc(\btheta_\Pi))\|_\eps \\
&\le \|\btheta_\Pi\|_\eps \,\gdivc \omega^{-1} \|\ROTZ\bestc(\btheta_\Pi)\|_\bnu,
\end{align*}
where we used the Cauchy--Schwarz inequality in the second line and 
the divergence conformity factor defined 
in~\eqref{eq_gamma_bX} in the third line
(recall that $\bestc(\btheta_\Pi)\in \bV\upc_{h0}$
by construction and observe that $\bestc(\btheta_\Pi) \in \bV\upc_{h0}(\ceqz)^\perp$
owing to~\eqref{eq:useful_pty} and~\eqref{eq:theta_Pi_perp_c}). 
Since 
\[
\|\ROTZ\bestc(\btheta_\Pi)\|_\bnu \le \tnorm{\bestc(\btheta_\Pi)}
\le \tnorm{\btheta_\Pi} = \omega\|\btheta_\Pi\|_\eps,
\]
owing to the definition of the $\tnorm{\SCAL}$-norm and the stability 
property~\eqref{eq:stab_best} of $\bestc$, we infer that
\[
|\Theta_1| \le \gdivc \|\btheta_\Pi\|_\eps^2.
\]
Furthermore, recalling that $\btheta_\Pi = \bE-\btheta_0-\bE_h$ and the definition of $\beeta$, we obtain
\[
\Theta_2 = (\btheta_\Pi,\beeta)_\eps - (\btheta_\Pi,\btheta_0-\bestc(\btheta_0))_\eps \qqe \Theta_{2,1}-\Theta_{2,2},
\]
where we used that $\bE_h-\bestc(\bE_h)=\bzero$. 
The Cauchy--Schwarz inequality gives
\[
|\Theta_{2,1}| = |(\btheta_\Pi,\beeta)_\eps| = |(\btheta_\Pi,\bPi\upc_0(\beeta))_\eps|
\le \|\btheta_\Pi\|_\eps\, \|\bPi\upc_0(\beeta)\|_\eps.
\]
Concerning $\Theta_{2,2}$, we have
\begin{align*}
|\Theta_{2,2}| &= |(\btheta_\Pi,\btheta_0-\bestc(\btheta_0))_\eps| = |(\btheta_\Pi,\bestc(\btheta_0))_\eps| 
= |(\btheta_\Pi,\bPi\upc_0(\bestc(\btheta_0)))_\eps| \\
&\le \|\btheta_\Pi\|_\eps\, \|\bPi\upc_0(\bestc(\btheta_0))\|_\eps \le \|\btheta_\Pi\|_\eps\, 
\gdivc\omega^{-1}
\|\ROTZ\bestc(\btheta_0)\|_\bnu \le \|\btheta_\Pi\|_\eps\, \gdivc\omega^{-1} \tnorm{\btheta_0}.
\end{align*}
Notice that we can again use the divergence conformity factor $\gdivc$ since $\btheta_0 \in
\Hrotzrotz^\perp$ implies by~\eqref{eq:useful_pty} that $\bestc(\btheta_0)\in 
\bV\upc_{h0}(\ceqz)^\perp$. 
Putting the above bounds on $\Theta_{2,1}$ and $\Theta_{2,2}$ together gives
\[
|\Theta_2| \le \|\btheta_\Pi\|_\eps\, \omega^{-1} \Big( \omega \|\bPi\upc_0(\beeta)\|_\eps
+ \gdivc \tnorm{\btheta_0}\Big).
\]  
The estimate~\eqref{eq:bnd_thet2_c} follows by combining the above bounds on $\Theta_1$ and $\Theta_2$.
\end{proof}

\begin{remark}[Lemma~\ref{lem:bnd_thet2_c}] 
Obviously, the estimate \eqref{eq:bnd_thet2_c} is meaningful only if $\gdivc<1$,
which holds true if the mesh size is small enough and/or the polynomial degree
is large enough; see Section~\ref{sec:bnd_app_fac} for further insight.
\end{remark}

\begin{lemma}[Bound on $\ROTZ\btheta_0$] \label{lem:bnd_rot_theta_c}
We have
\begin{equation} \label{eq:bnd_rot_theta_c}
\|\ROTZ\btheta_0\|_\bnu^2\le \tnorm{(I-\bPi\upc_0)(\beeta)}^2
+ (2\gdivc+3\gpc^2)\tnorm{\be}^2 + 2\gdivc\omega^2\|\btheta_\Pi\|_\eps^2.
\end{equation}
\end{lemma}

\begin{proof}
Recalling that $\btheta_0=(I-\bPi\upc_0)(\be)$, a straightforward calculation shows that
\begin{align*}
b(\btheta_0,\btheta_0) &= b(\btheta_0,(I-\bPi\upc_0)(\be)) \\
&=b(\btheta_0,(I-\bPi\upc_0)(\beeta))-b(\btheta_0,(I-\bPi\upc_0)(\bE_h-\bestc(\bE))) \\
&=b(\btheta_0,(I-\bPi\upc_0)(\beeta))-b(\be,(I-\bPi\upc_0)(\bE_h-\bestc(\bE))) \\
&\le \tnorm{\btheta_0} \, \tnorm{(I-\bPi\upc_0)(\beeta)}-b(\be,(I-\bPi\upc_0)(\bE_h-\bestc(\bE))),
\end{align*}
where we used that $b(\btheta_\Pi,(I-\bPi\upc_0)(\SCAL)) = 0$ on the third line
and the boundedness of $b$ on the fourth line. Focusing on the second term on the
right-hand side, we notice using Galerkin's orthogonality that
\begin{align*}
-b(\be,(I-\bPi\upc_0)(\bE_h-\bestc(\bE))) &=
b(\be,\bPi\upc_0(\bE_h-\bestc(\bE))) \\
&=-\omega^2 (\btheta_\Pi,\bPi\upc_0(\bE_h-\bestc(\bE)))_\eps \\
&\le \omega^2 \|\btheta_\Pi\|_\eps\, \|\bPi\upc_0(\bE_h-\bestc(\bE))\|_\eps.
\end{align*}
For all $\bv_h\in \bV\upc_{h0}(\ceqz)$, we have
\begin{align*}
\omega^2(\bE_h-\bestc(\bE),\bv_h)_\eps
=
-b(\bE_h,\bv_h)-b^+(\bestc(\bE),\bv_h)
&=
-b(\bE,\bv_h)-b^+(\bE,\bv_h)
\\
&= \omega^2(\bE,\bv_h)-\omega^2(\bE,\bv_h) = 0,
\end{align*}
where the first and third equalities follow from the fact that $\bv_h$ is curl-free,
and the second from the definition of $\bestc$ and Galerkin's orthogonality.
This shows that
\[
\bE_h-\bestc(\bE) \in \bV\upc_{h0}(\ceqz)^\perp.
\]
Using the divergence conformity factor $\gdivc$ defined in~\eqref{eq_gamma_bX} then yields
\[
|b(\be,(I-\bPi\upc_0)(\bE_h-\bestc(\bE)))| \le 
\gdivc \omega \|\btheta_\Pi\|_\eps \|\ROTZ(\bE_h-\bestc(\bE))\|_\bnu.
\]
Invoking the triangle inequality and the stability property~\eqref{eq:stab_best} of $\bestc$ gives
\begin{align*}
\|\ROTZ(\bE_h-\bestc(\bE))\|_\bnu & \le \|\ROTZ(\bE_h-\bE)\|_\bnu+\|\ROTZ(\bE-\bestc(\bE))\|_\bnu \\
& \le \tnorm{\be}+\tnorm{\bE-\bestc(\bE)} 
\le 2\tnorm{\be},
\end{align*}
since $\tnorm{\bE-\bestc(\bE)} \le \tnorm{\bE-\bE_h}=\tnorm{\be}$ 
by definition of the best-approximation operator $\bestc$. 
Putting everything together gives
\[
b(\btheta_0,\btheta_0) \le \tnorm{\btheta_0} \, \tnorm{(I-\bPi\upc_0)(\beeta)} + 2\gdivc \omega \|\btheta_\Pi\|_\eps\tnorm{\be}.
\]
As a result, we have
\begin{align*}
\|\ROTZ\btheta_0\|_\bnu^2 &= b(\btheta_0,\btheta_0) + \omega^2\|\btheta_0\|_\eps^2 \\
& \le \tnorm{\btheta_0} \, \tnorm{(I-\bPi\upc_0)(\beeta)} + 2\gdivc \omega \|\btheta_\Pi\|_\eps\tnorm{\be}+ \omega^2\|\btheta_0\|_\eps^2.
\end{align*}
Invoking Young's inequality for the first and second terms on the right-hand side 
and using that $\tnorm{\btheta_0}^2=\|\ROTZ\btheta_0\|_\bnu^2+\omega^2\|\btheta_0\|_\eps^2$, we infer that
\[
\|\ROTZ\btheta_0\|_\bnu^2 \le \tnorm{(I-\bPi\upc_0)(\beeta)}^2 + 
2\gdivc\tnorm{\be}^2 + 2\gdivc\omega^2 \|\btheta_\Pi\|_\eps^2 + 3\omega^2\|\btheta_0\|_\eps^2.
\]
The assertion now follows from the bound on $\btheta_0$ established in Lemma~\ref{lem:bnd_thet1_c}.
\end{proof}

\subsection{Error estimate}

We are now ready to establish our main error estimate.

\begin{theorem}[A priori error estimate and discrete well-posedness] \label{th:est_err_c}
Assume that $\gdivc \le 1$. 
The following holds:
\begin{equation} \label{eq:opt_err_c}
(1-15\gdivc-4\gpc^2)\tnorm{\bE-\bE_h}
\le
\min_{\bv_h \in \bV_{h0}\upc}\tnorm{\bE-\bv_h}.
\end{equation}
Consequently, if the mesh size is small enough and/or the polynomial
degree is large enough so that $15\gdivc+4\gpc^2<1$,
the discrete problem~\eqref{eq:disc_pb_c} is well-posed. 
\end{theorem}

\begin{proof}
(1) In this first step, we establish some preliminary bounds.
Since $\tnorm{\btheta_0}\le \tnorm{\be}$ (because $\tnorm{\be}^2=\tnorm{\btheta_0}^2 +
\omega^2\|\btheta_\Pi\|_\eps^2$), 
the estimate \eqref{eq:bnd_thet2_c} implies that
\begin{equation} \label{eq:bnd_theta22a_c}
(1-\gdivc) \omega \|\btheta_\Pi\|_\eps \le \omega \|\bPi\upc_0(\beeta)\|_\eps 
+ \gdivc\tnorm{\be}.
\end{equation}
Moreover, we have
\[
\omega \|\bPi\upc_0(\beeta)\|_\eps \le \omega \|\beeta\|_\eps \le \tnorm{\beeta}
\le \tnorm{\bE-\bE_h} = \tnorm{\be}.
\]
Squaring~\eqref{eq:bnd_theta22a_c} (recall that $\gdivc\le 1$ by assumption) and using 
the above bound in the double product, we obtain
\begin{align}
(1-\gdivc)^2 \omega^2 \|\btheta_\Pi\|_\eps^2 \le {}& \omega^2 \|\bPi\upc_0(\beeta)\|_\eps^2 
+ (2\gdivc+\gdivc^2)\tnorm{\be}^2 \nonumber \\
\le {}& \omega^2 \|\bPi\upc_0(\beeta)\|_\eps^2 + 3\gdivc\tnorm{\be}^2, \label{eq:bnd_prep}
\end{align}
where the last bound follows from $\gdivc\le1$. Since $\omega \|\bPi\upc_0(\beeta)\|_\eps \le \tnorm{\be}$ and
$\gdivc\le1$, \eqref{eq:bnd_prep} implies that 
\begin{equation} \label{eq:bnd_prepp}
(1-\gdivc)^2 \omega^2 \|\btheta_\Pi\|_\eps^2 \le 4 \tnorm{\be}^2.
\end{equation}

(2) We are now ready to prove~\eqref{eq:opt_err_c}.
Multiplying the estimate~\eqref{eq:bnd_rot_theta_c} from Lemma~\ref{lem:bnd_rot_theta_c} 
by $(1-\gdivc)^2$ (which is $\le1$) and using~\eqref{eq:bnd_prepp} gives 
\begin{align}
(1-\gdivc)^2\|\ROTZ\btheta_0\|_\bnu^2 
\le {}& \tnorm{(I-\bPi\upc_0)(\beeta)}^2 + (2\gdivc+3\gpc^2)\tnorm{\be}^2 + 2\gdivc(1-\gdivc)^2\omega^2\|\btheta_\Pi\|_\eps^2
\nonumber \\
\le {}&\tnorm{(I-\bPi\upc_0)(\beeta)}^2 + (10\gdivc+3\gpc^2)\tnorm{\be}^2. \label{eq:prelim_bnd_c}
\end{align}
Since $\tnorm{\be}^2 = \omega^2\|\btheta_0\|_\eps^2 + \omega^2\|\btheta_\Pi\|_\eps^2 + \|\ROTZ\btheta_0\|_\bnu^2$, we infer that
\begin{align}
(1-\gdivc)^2\tnorm{\be}^2 \le{}& \omega^2\|\btheta_0\|_\eps^2 + (1-\gdivc)^2\omega^2\|\btheta_\Pi\|_\eps^2 + (1-\gdivc)^2\|\ROTZ\btheta_0\|_\bnu^2 \nonumber \\
\le{}& \gpc^2 \tnorm{\be}^2 + \omega^2 \|\bPi\upc_0(\beeta)\|_\eps^2 + 3\gdivc\tnorm{\be}^2 \nonumber \\ 
&+\tnorm{(I-\bPi\upc_0)(\beeta)}^2 + (10\gdivc+3\gpc^2)\tnorm{\be}^2 \nonumber \\
= {}& \tnorm{\beeta}^2 + (13\gdivc+4\gpc^2)\tnorm{\be}^2, \label{eq:bnd_e_eta_e}
\end{align}
where the second bound follows from Lemma~\ref{lem:bnd_thet1_c}, \eqref{eq:bnd_prep}, and \eqref{eq:prelim_bnd_c}, and the last equality follows from $\tnorm{\beeta}^2=\tnorm{(I-\bPi\upc_0)(\beeta)}^2+\omega^2 \|\bPi\upc_0(\beeta)\|_\eps^2$.
The error estimate~\eqref{eq:opt_err_c} follows by observing that $(1-\gdivc)^2 \ge 1-2\gdivc$.

(3) If the mesh size is small enough so that
$15\gdivc + 4\gpc^2<1$,
the error estimate~\eqref{eq:opt_err_c} implies the
uniqueness of the discrete solution. Existence then follows from the fact that \eqref{eq:disc_pb_c} amounts to a square linear system.
\end{proof}

\begin{remark}[Asymptotic optimality]
Notice that in \eqref{eq:opt_err_c}, we have $\gdivc \to 0$ and $\gpc \to 0$
as $h\to0$ or $k\to\infty$. Hence, we have
\begin{equation*}
\tnorm{\bE-\bE_h}
\leq
(1+\theta(h)) \min_{\bv_h \in \bV_{h0}\upc}\tnorm{\bE-\bv_h},
\end{equation*}
with $\lim_{h/(k+1) \to 0} \theta(h) = 0$.
\end{remark}

\subsection{Inf-sup stability}

We are now ready to establish our main stability result.

\begin{theorem}[Inf-sup stablity]
\label{th:inf_sup_c}
We have
\begin{equation} \label{eq:inf_sup_c}
\min_{\substack{\bv_h \in \bV_{h0}\upc \\ \tnorm{\bv_h} = 1}}
\max_{\substack{\bw_h \in \bV_{h0}\upc \\ \tnorm{\bw_h} = 1}}
|b(\bv_h,\bw_h)|
\geq
\frac{1-2(\gdivc^2+\gpc)}{1 + 2\bst}.
\end{equation}
\end{theorem}

\begin{proof}
We adapt to the discrete setting the arguments of the proof of Lemma~\ref{lem:infsup}.  
Let $\bv_h \in \bV_{h0}\upc$ and set $\bv_h=\bv_{h0}+\bv_{h\Pi}$ with
$\bv_{h0}\eqq (I-\bPi\upc_{h0})(\bv_h) \in \bX\upc_{h0}$
and $\bv_{h\Pi}\eqq \bPi\upc_{h0}(\bv_h)$.

(1) In this first step, we gain control on $\omega\|\bv_{h0}\|_\eps$.
Since $\bv_{h0}$ is (loosely speaking)
discretely divergence-free, but not pointwise divergence-free,
we need to consider a further decomposition of $\bv_{h0}$. 
Let us set $\bv_{h0}=\bphi_0+\bphi_\Pi$ with
$\bphi_0\eqq (I-\bPi\upc_0)(\bv_{h0})$ and $\bphi_\Pi\eqq \bPi\upc_0(\bv_{h0})$. 
Notice that
\[
\omega\|\bphi_0\|_\eps \le \omega\|\bv_{h0}\|_\eps \le \omega\|\bv_{h}\|_\eps \le \tnorm{\bv_h}.
\]
Let $\bxi_0\in \bX\upc_0$ be the unique adjoint solution such that
$b(\bw,\bxi_0)=\omega^2(\bw,\bphi_0)_\eps$ for all $\bw\in \bV_0\upc$ (notice that $\bphi_0
\in\Hrotzrotz^\perp$). We have
\[
\tnorm{\bxi_0} \le \bst \omega\|\bphi_0\|_\eps\le \bst \omega \|\bv_{h0}\|_\eps.
\]
Let us set $\bxi_{h0}\eqq \bestc(\bxi_0)$. Then $\bxi_{h0}\in \bV_{h0}\upc$ by definition,
and $\bxi_{h0}\in \bV\upc_{h0}(\ceqz)^\perp$ by~\eqref{eq:useful_pty}
since $\bxi_0\in \Hrotzrotz^\perp$. 
Moreover, owing to~\eqref{eq:stab_best}, we have 
$\tnorm{\bxi_{h0}} \le \tnorm{\bxi_0}$. Using these properties, we infer that
\begin{align*}
b(\bv_h,\bxi_{h0}) &= b(\bv_{h0},\bxi_{h0})+b(\bv_{h\Pi},\bxi_{h0}) = b(\bv_{h0},\bxi_{h0}) \\
&=b(\bv_{h0},\bxi_{0})+b(\bv_{h0},\bxi_{h0}-\bxi_0) \\
&\ge \omega^2\|\bphi_0\|_\eps^2 - \gpc \tnorm{\bv_h}^2,
\end{align*} 
since $b(\bv_{h0},\bxi_{0})=\omega^2(\bv_{h0},\bphi_0)_\eps=\omega^2\|\bphi_0\|_\eps^2$ and
\[
|b(\bv_{h0},\bxi_{h0}-\bxi_0)| \le \tnorm{\bv_{h0}}\, \tnorm{\bxi_{h0}-\bxi_0} \le
\tnorm{\bv_{h0}}\, \gpc \omega\|\bphi_0\|_\eps \le \gpc\tnorm{\bv_h}^2,
\]
owing to the boundedness of the bilinear form $b$, the above bound on $\omega\|\bphi_0\|_\eps$, and since $\tnorm{\bv_{h0}}\le\tnorm{\bv_h}$. Moreover, using the divergence conformity factor,
we infer that
\[
\omega\|\bphi_\Pi\|_\eps^2 = \omega^2 \|\bPi\upc_0(\bv_{h0})\|_\eps^2 \le \gdivc^2\|\ROTZ\bv_{h0}\|_\bnu^2 \le \gdivc^2 \tnorm{\bv_h}^2.
\]
Since $\|\bv_{h0}\|_\eps^2=\|\bphi_0\|_\eps^2+\|\bphi_\Pi\|_\eps^2$, putting everything together gives
\begin{equation} \label{eq:lower_infsup_c}
b(\bv_h,\bxi_{h0}) \ge \omega^2\|\bv_{h0}\|_\eps^2 - (\gdivc^2+\gpc)\tnorm{\bv_h}^2.
\end{equation}

(2) Since $b(\bv_h,\bv_{h0}-\bv_{h\Pi})=\tnorm{\bv_h}^2 - 2\omega^2\|\bv_{h0}\|_\eps^2$, using \eqref{eq:lower_infsup_c} yields
\[
b(\bv_h,\bv_{h0}+2\bxi_{h0}-\bv_{h\Pi}) \ge \tnorm{\bv_h}^2 -2(\gdivc^2+\gpc)\tnorm{\bv_h}^2.
\]
Moreover, using the same arguments as in the proof of Lemma~\ref{lem:infsup} and recalling that $\tnorm{\bxi_{h0}}\le \tnorm{\bxi}$, we obtain
\[
\tnorm{\bv_{h0}+2\bxi_{h0}-\bv_{h\Pi}}^2 = \tnorm{\bv_{h0}+2\bxi_{h0}}^2 + \tnorm{\bv_{h\Pi}}^2 \le (1+2\bst)^2 \tnorm{\bv_h}^2.
\] 
Since $\bv_{h0}+2\bxi_{h0}-\bv_{h\Pi}\in \bV\upc_{h0}$, this concludes the proof.
\end{proof}

\begin{remark}[Discrete inf-sup constant]
\label{rem:disc_infsup}
Since $\gpc$ and $\gdivc$ tend to zero as the mesh size is small enough
and/or the polynomial degree is large enough, the discrete inf-sup constant
appearing on the left-hand side of~\eqref{eq:inf_sup_c} tends to $(1+2\bst)^{-1}$.
Recall from Remark~\ref{rem:exact_infsup} that this quantity corresponds, up to
a factor of two at most, to the inf-sup constant of the bilinear form $b$ in the
continuous setting.
\end{remark}

\section{Bound on approximation and divergence conformity factors}
\label{sec:bnd_app_fac}

In this section, we bound the two (nondimensional) quantities 
introduced in Section~\ref{sec:def_factors_c} and used in Section~\ref{sec:conforming}: 
the approximation factor $\gpc$
and the divergence conformity factor $\gdivc$.
To this purpose, we consider the commuting quasi-interpolation operator 
$\calJ\upc_{h0}: \bL^2(\Dom) \to \bV\upc_{h0}$ 
and $\calJ\upd_{h0}: \bL^2(\Dom) \to \bV\upd_{h0}$ (the Raviart--Thomas
finite element space of order $k\ge0$ satisfying zero normal boundary conditions); see
\cite{Schoberl:01,ArnFW:06,Christiansen:07,ChrWi:08} and 
also \cite[Chap.~22-23]{EG_volI}. Both operators are bounded
in $\bL^2(\Dom)$, they are projections, and they satisfy the commuting 
property $\ROT(\calJ\upc_{h0}(\bv))=\calJ\upd_{h0}(\ROT\bv)$ for all $\bv\in\bL^2(\Dom)$.

For positive real numbers $A$ and $B$, we abbreviate as $A\lesssim B$
the inequality $A\le CB$ with a generic constant $C$ whose value can change at each
occurrence as long as it is independent of the mesh size, the frequency parameter $\omega$,
and, whenever relevant, any function involved in the bound. The constant $C$ can depend on
the shape-regularity of the mesh and the polynomial degree $k$ as well as on the
domain $\Omega$ and on the coefficients $\eps$ and $\bnu$.

We introduce the notation
\begin{equation}
\epsmax
:=
\max_{\bx \in \Dom} \max_{\substack{\bu \in \mathbb R^d \\ |\bu| = 1}}
\max_{\substack{\bv \in \mathbb R^d \\ |\bv| = 1}}
\eps(\bx) \bu \cdot \bv
\qquad
\epsmin
:=
\min_{\bx \in \Dom} \min_{\substack{\bu \in \mathbb R^d \\ |\bu| = 1}}
\eps(\bx) \bu \cdot \bu
\end{equation}
and define $\numax$ and $\numin$ similarly. Then,
$\velmin = \sqrt{\numin/\epsmax}$ 
stands for the minimum wavespeed in the domain.

\subsection{Piecewise smooth coefficients}

For the sake of simplicity, we start by assuming that the
coefficients are piecewise smooth in $\Dom$. Then, the following
regularity results from \cite{costabel_dauge_nicaise_1999a,Jochmann_maxwell_1999,BoGuL:13}
will be useful: there exists $s \in (0,1]$
such that, for all $\bv \in \Hrotz$ with $\DIV(\eps\bv) = 0$ and all
$\bw \in \Hdivzdivz$ with $\bnu \bw \in \Hrot$, we have
$\bv,\bw \in \bH^s(\Dom)$ with the estimates
\begin{equation}
\label{eq:extra_regularity}
|\bv|_{\bH^s(\Dom)}
\lesssim
\ell_\Dom^{1-s}
\numin^{-\frac12}\|\ROTZ \bv\|_\bnu,
\qquad
|\bw|_{\bH^s(\Dom)}
\lesssim
\ell_\Dom^{1-s}
\frac{1}{\velmin} \numin^{-\frac12}
\|\ROT (\bnu \bw)\|_{\eps^{-1}}.
\end{equation}
If $\Dom$ is convex and $\eps$ and $\bnu$
are (globally) Lipschitz continuous, we can take $s = 1$.


\begin{lemma}[Bound on approximation factor]
\label{lemma_approximation_factor_c}
Let $\gpc$ be defined in~\eqref{eq:def_gamma_c}. The following holds:
\begin{equation} \label{eq:regularity_c}
\gpc
\lesssim
(1+\bst)\left (\frac{\omega\ell_\Dom}{\velmin}\right )^{1-s}
\left (\frac{\omega h}{\velmin}\right )^s,
\end{equation}
with the stability constant $\bst$ defined in~\eqref{eq:def_bst}.
\end{lemma}

\begin{proof}
Let $\btheta \in \Hrotzrotz^\perp$ and let $\bzeta_\btheta \in \bX\upc_0$ solve the adjoint 
problem~\eqref{eq:adjoint}.
On the one hand, invoking~\eqref{eq:extra_regularity}, using the stability constant $\bst$,
we infer that
\[
|\bzeta_\btheta|_{\bH^s(\Dom)}
\lesssim
\ell_\Dom^{1-s} \numin^{-\frac12}\|\ROTZ \bzeta_\btheta\|_\bnu
\le
\ell_\Dom^{1-s} \numin^{-\frac12}\tnorm{\bzeta_\btheta}
\lesssim
\bst \ell_\Dom^{1-s} \omega^{-1} \numin^{-\frac12}\|\btheta\|_\eps.
\]
Invoking the approximation properties of $\calJ\upc_{h0}$
leads to
\begin{align}
\omega^2 \|\bzeta_\btheta-\calJ\upc_{h0}(\bzeta_\btheta)\|_\eps
&\leq
\omega^2\epsmax^{\frac12}
\|\bzeta_\btheta-\calJ\upc_{h0}(\bzeta_\btheta)\|
\nonumber \\
& \lesssim
\omega^2
h^s \epsmax^{\frac12}|\bzeta_\btheta|_{\bH^s(\Dom)}
\lesssim
\bst
\left (\frac{\omega\ell_\Dom}{\velmin}\right )^{1-s}
\left (\frac{\omega h}{\velmin}\right )^s
\|\btheta\|_\eps. \label{tmp_gpc_l2}
\end{align}
On the other hand, we have
$\eps^{-1} \ROT(\bnu\ROTZ\bzeta_\btheta) = \btheta + \omega^2 \bzeta_\btheta$,
so that
\begin{equation*}
\|\ROT(\bnu\ROTZ\bzeta_\btheta)\|_{\eps^{-1}}
=
\|\eps^{-1} \ROT(\bnu\ROTZ\bzeta_\btheta)\|_{\eps}
\leq
\|\btheta\|_\eps + \omega^2 \|\bzeta_\btheta\|_\eps
\leq
(1+\bst)\|\btheta\|_\eps.
\end{equation*}
Since $\bw\eqq \ROTZ \bzeta_\btheta \in \Hdivzdivz$ with $\bnu\bw\in\Hrot$, 
we can again invoke \eqref{eq:extra_regularity}, giving
\begin{equation*}
|\ROTZ\bzeta_\btheta|_{\bH^s(\Dom)}
\lesssim
\ell_\Dom^{1-s} \frac{\numin^{-\frac12}}{\velmin}
\|\ROT(\bnu\ROTZ\bzeta_\btheta)\|_{\eps^{-1}}
\leq
(1+\bst) \ell_\Dom^{1-s} \frac{\numin^{-\frac12}}{\velmin} \|\btheta\|_\eps.
\end{equation*}
Owing to the commuting property $\ROTZ\calJ\upc_{h0}(\SCAL)=\calJ\upd_{h0}(\ROTZ\SCAL)$
where $\calJ\upd_{h0}$ is the commuting quasi-interpolation operator mapping onto the 
Raviart--Thomas finite element space with zero normal component on the boundary, 
we infer that
\begin{align}
\omega \|\ROTZ(\bzeta_\btheta-\calJ\upc_{h0}(\bzeta_\btheta))\|_\bnu
&= \omega \|\ROTZ\bzeta_\btheta-\calJ\upd_{h0}(\ROTZ\bzeta_\btheta)\|_\bnu \nonumber \\
&\lesssim
\omega \numax^{\frac12} h^s|\ROTZ\bzeta_\btheta|_{\bH^s(\Dom)}
\nonumber \\
&\lesssim
(1+\bst)
\left (\frac{\omega \ell_\Dom}{\velmin}\right )^{1-s}
\left (\frac{\omega h}{\velmin}\right )^{s}
\|\btheta\|_\eps. \label{tmp_gpc_rot}
\end{align}
(Recall that the ratio
$\numax/\numin$ can be hidden in the generic constant $C$.)
The conclusion follows from \eqref{tmp_gpc_l2} and \eqref{tmp_gpc_rot}.
\end{proof}

\begin{lemma}[Bound on divergence conformity factor] \label{lem:div_conf_conf}
\label{lemma_gdivc_smooth}
Let $\gdivc$ be defined in~\eqref{eq_gamma_bX}. The following holds:
\begin{equation}
\gdivc
\lesssim
\left (\frac{\omega\ell_\Dom}{\velmin}\right )^{1-s}
\left (\frac{\omega h}{\velmin}\right )^s.
\end{equation}
\end{lemma}

\begin{proof}
(1) Let $\bv_h \in \bX\upc_{h0}= \bV\upc_{h0}\cap \bV\upc_{h0}(\ceqz)^\perp$. Let us write
\begin{equation*}
\bv_h = \bw + \bPi\upc_0(\bv_h),
\end{equation*}
with $\bw \eqq (I-\bPi\upc_0)(\bv_h)$. By construction, 
$\bw\in \Hrotzrotz^\perp$, and we have $\bw\in \Hrotz$ since $\bv_h\in \bV\upc_{h0}
\subset \Hrotz$; hence, $\bw\in \bX\upc_0$. Invoking~\eqref{eq:extra_regularity}
and observing that $\ROTZ \bw = \ROTZ \bv_h$, we infer that
\[
|\bw|_{\bH^s(\Dom)} \lesssim \ell_\Dom^{1-s} \numin^{-\frac12} \|\ROTZ\bv_h\|_\bnu.
\]
Moreover, we have $\bPi\upc_{h0}(\bPi\upc_0(\bv_h))
= \bPi\upc_{h0}(\bv_h) = \bzero$ since $\bv_h\in \bV\upc_{h0}(\ceqz)^\perp$;
hence, $\bPi\upc_0(\bv_h)\in \bV\upc_{h0}(\ceqz)^\perp$ as well.

(2) Recall the commuting quasi-interpolation operator 
$\calJ\upc_{h0}: \bL^2(\Dom) \to \bV\upc_{h0}$. 
Since $\calJ\upc_{h0}$ leaves $\bV\upc_{h0}$ pointwise invariant, we have
$(I-\calJ\upc_{h0})(\bv_h)=\bzero$, so that
\begin{equation*}
(I-\calJ\upc_{h0})(\bPi\upc_0(\bv_h)) = -(I - \calJ\upc_{h0})(\bw).
\end{equation*}
Moreover, since $\bPi\upc_0(\bv_h)$ is curl-free by construction, the commuting property
of $\calJ\upc_{h0}$ implies that 
\[
\calJ\upc_{h0}(\bPi\upc_0(\bv_h)) \in 
\bV\upc_{h0}(\ceqz).
\] 
Recalling that $\bPi\upc_0(\bv_h)\in 
\bV\upc_{h0}(\ceqz)^\perp$, we infer that
\begin{align*}
\|\bPi\upc_0(\bv_h)\|_\eps^2 = (\bPi\upc_0(\bv_h),\bPi\upc_0(\bv_h))_\eps
&= (\bPi\upc_0(\bv_h),\bPi\upc_0(\bv_h)-\calJ\upc_{h0}(\bPi\upc_0(\bv_h)))_\eps \\
&= -(\bPi\upc_0(\bv_h),\bw-\calJ\upc_{h0}(\bw))_\eps.
\end{align*}
The Cauchy--Schwarz inequality together with the approximation properties of 
$\calJ\upc_{h0}$ gives
\[
\|\bPi\upc_0(\bv_h)\|_\eps \lesssim h^s \epsmax^{\frac12}|\bw|_{\bH^s(\Dom)},
\]
and we conclude using the above bound on $|\bw|_{\bH^s(\Dom)}$.
\end{proof}

\begin{remark}[Bound on $\gdivc$]
The above proof can be rewritten as the following statement:
\begin{equation} \label{eq:gamma_calJ}
\gdivc \le 
\gamma_\calJ := \sup_{\substack{\bw \in \bX\upc_0 \\ \|\ROTZ \bw\|_\bnu = 1}}
\omega \|\bw-\calJ\upc_{h0}(\bw)\|_\eps \lesssim 
\left (\frac{\omega\ell_\Dom}{\velmin}\right )^{1-s}
\left (\frac{\omega h}{\velmin}\right )^s.
\end{equation}
This shows that $\gdivc$ is bounded by an approximation factor on $\bX\upc_0$
using the commuting quasi-interpolation operator $\calJ\upc_{h0}$.
Notice that only the rightmost bound uses~\eqref{eq:extra_regularity}.
\end{remark}

\begin{remark}[Convex domain]
For a convex domain $\Dom$, the factors are bounded as
\begin{equation*}
\gpc
\lesssim
(1+\bst)
\frac{\omega h}{\velmin},
\qquad
\gdivc
\lesssim
\frac{\omega h}{\velmin}.
\end{equation*}
The quantity $(\omega h)/\velmin$ is inversly proportinal to the
(minimal) number of mesh elements per wavelenth. It is therefore reasonable
to assume that $\gdivc \lesssim 1$. We also see that
$\gpc$ is typically not bounded for all frequencies assuming a
constant number of elements per wavelength, since $\bst$
can be large. This is the standard manifestation of dispersion errors,
also known in this context as pollution effect. This is completely
standard, and also happens in the (simpler) case of Helmholtz problems.
It is interesting to notice that the constraint that $\gdivc$ is small,
which is specific to Maxwell's equations, is less restrictive than
the constraint that $\gpc$ is small, which is common to Maxwell and
Helmholtz equations.
\end{remark}


\begin{remark}[Reduced dispersion for high-order elements]
When the domain $\Dom$ and the coefficients are smooth, it is shown in \cite{ChaVe:22} that
\begin{equation*}
\gpc
\lesssim
\frac{\omega h}{\velmin} + (1+\bst) \left (\frac{\omega h}{\velmin}\right )^{(k+1)},
\end{equation*}
so that $\gpc$ is small if
\begin{equation}
\label{eq_dof_reduction}
\frac{\omega h}{\velmin}
\lesssim
\bst^{-1/(k+1)}.
\end{equation}
For large frequencies (or frequencies close to resonant frequencies),
$\bst$ becomes large, so that the number of elements per wavelength
needs to be increased. Nevertheless, \eqref{eq_dof_reduction} expresses
that the required increase is less important for higher order elements,
which corresponds to numerical observations. It is also expected
that such a result remains true for general domains and piecewise smooth
coefficients if the mesh is suitably refined locally, but this
claim has only been established for two-dimensional problems in
\cite{chaumontfrelet_nicaise_2020a}. We finally refer the reader to
\cite{melenk_sauter_2021} where stronger results explicit in the polynomial
degree $k$ are established, but under stronger assumptions on the domain.
\end{remark}

\begin{remark}[$k$ convergence]
When $s=1$, we can consider the interpolation operators
from \cite{MelRo:20}, instead of the quasi-interpolation
operators $\calJ\upc_{h0}$ and $\calJ\upd_{h0}$ considered above,
thereby showing that $\gpc$ and $\gdivc$ tend to zero (and optimally so)
as $k$ is increased (and $h$ is fixed). More generally, a similar strategy
can be employed as long as a piecewise version of the regularity regularity shift
in \eqref{eq:extra_regularity} holds true with $s > 1/2$. In such a case, we can use,
e.g., the interpolation operators from \cite{DeBuffa05}.
\end{remark}

\subsection{General coefficients}

Here, we consider general coefficients $\eps$ and $\bnu$ for
which \eqref{eq:extra_regularity} may not hold for any $s > 0$.

\begin{lemma}[Convergence of approximation factor]
We have $\gpc \to 0$ as $h \to 0$.
\end{lemma}

\begin{proof}
Let us set
\begin{align*}
\calB_\eps
&:=
\left \{
\bv \in \Hrotz \; | \;
\DIV(\eps\bv) = 0, \;
\|\ROTZ \bv\|_{\bmu}
\leq \bst \right \}, \\
\calB_\bnu
&:=
\left \{
\bw \in \Hdivzdivz \; | \;
\|\ROT (\bnu\bw)\|_{\eps^{-1}}
\leq 1+\bst \right \}.
\end{align*}
Owing to the definition~\eqref{eq:def_bst} of the stability constant
$\bst$, if $\btheta \in \Hrotzrotz^\perp$
with $\|\btheta\|_\eps = 1$ and $\bxi_\btheta \in \bV_0\upc$ satisfy
$b(\bw,\bxi_\btheta) = (\bw,\btheta)_\eps$ for all $\bw \in \bV_0\upc$,
we have $\bxi_\btheta \in \calB_\eps$ and $\bPhi_\btheta := \ROTZ
\bxi_\btheta \in \calB_\bnu$.

Owing to the compact injections established in \cite[Theorem 2.2]{Weber:80},
given any $\delta > 0$, there exists a finite
number $N_\delta$ of functions $\bv_j \in \calB_\eps$ and $\bw_\ell \in \calB_\bnu$
such that, for all $\bv \in \calB_\eps$ and all $\bw \in \calB_\bnu$, there exist
indices $j,\ell \in \{1{:} N_\delta\}$ such that
\begin{equation*}
\|\bv-\bv_j\|_\eps \leq \delta, \qquad \|\bw-\bw_\ell\|_\bnu \leq \delta.
\end{equation*}
Furthermore, the density of $\bC^\infty_{\rm c}(\Dom)$ in $\bL^2(\Dom)$ implies that
we can find $\widetilde \bv_j,\widetilde \bw_\ell \in \bC^\infty_{\rm c}(\Dom)$
such that
\begin{equation*}
\|\bv-\widetilde \bv_j\|_\eps \leq 2\delta, \qquad \|\bw-\widetilde \bw_\ell\|_\bnu \leq 2\delta.
\end{equation*}

We then write that
\begin{equation*}
\min_{\bv\upc_h \in \bV_{h0}\upc}
\tnorm{\bxi_{\btheta}-\bv\upc_h}^2
\leq
\tnorm{\bxi_{\btheta}-\calJ\upc_{h0}(\bxi_{\btheta})}^2
=
\omega^2\|\bxi_{\btheta}-\calJ\upc_{h0}(\bxi_{\btheta})\|^2_\eps
+
\|\bPhi_{\btheta}-\calJ\upd_{h0}(\bPhi_{\btheta})\|^2_\bnu.
\end{equation*}
Invoking the triangle inequality and the above bounds (with $\bv:=\bxi_{\btheta}$) gives
\begin{align*}
\|\bxi_{\btheta}-\calJ\upc_{h0}(\bxi_{\btheta})\|_\eps
&\leq
\|\bxi_{\btheta}-\widetilde \bv_j\|_\eps
+
\|\widetilde \bv_j-\calJ\upc_{h0}(\widetilde \bv_j)\|_\eps
+
\|\calJ\upc_{h0}(\widetilde \bv_j-\bxi_{\btheta})\|_\eps
\\
&\lesssim
\|\bxi_{\btheta}-\widetilde \bv_j\|_{\eps}
+
\|\widetilde \bv_j-\calJ\upc_{h0}(\widetilde \bv_j)\|_\eps
\leq
2\delta
+
\|\widetilde \bv_j-\calJ\upc_{h0}(\widetilde \bv_j)\|_\eps,
\end{align*}
where we used the $\bL^2$-stability of $\calJ\upc_{h0}$.  Similarly, we obtain
\begin{equation*}
\|\bPhi_{\btheta}-\calJ\upd_{h0}(\bxi_{\btheta})\|_\bnu
\lesssim
\|\bPhi_{\btheta}-\widetilde \bw_\ell\|_\bnu
+
\|\widetilde \bw_\ell-\calJ\upd_{h0}(\widetilde \bw_\ell)\|_\bnu
\leq
2\delta
+
\|\widetilde \bw_\ell-\calJ\upd_{h0}(\widetilde \bw_\ell)\|_\bnu.
\end{equation*}
Since the functions 
$\widetilde \bv_j$ and $\widetilde \bw_\ell$ are in finite number and smooth,
we can assume that, if the mesh is sufficiently refined,
\begin{equation*}
\|\widetilde \bv_j-\calJ\upc_{h0}(\widetilde \bv_j)\|_\eps
\leq
\delta,
\qquad
\|\widetilde \bw_j-\calJ\upd_{h0}(\widetilde \bw_j)\|_\bnu
\leq
\delta.
\end{equation*}
This completes the proof since $\delta>0$ is arbitrary.
\end{proof}

\begin{lemma}[Convergence of divergence conformity factor]
We have $\gdivc \to 0$ as $h \to 0$.
\end{lemma}

\begin{proof}
We use the bound $\gdivc \le \gamma_\calJ$ with $\gamma_\calJ$ defined in~\eqref{eq:gamma_calJ}.
Thus, it suffices to show that $\gamma_\calJ\to0$ as $h\to0$.

We consider the unit ball
$\calB := \{ \bw \in \Hrotz; \DIV(\eps\bw) = 0, \; \|\ROTZ\bw\|_\bnu \leq 1\}$.
It is established in \cite[Theorem 2.2]{Weber:80} that the embedding
$\Hrotz \cap \Hdiveps \hookrightarrow \bL^2(\Dom)$ is compact. As a
result, given any $\delta > 0$, there exists a finite number $N_\delta$
of elements $\bv_j \in \calB$ such that, for all $\bw \in \calB$,
$\|\bw-\bv_j\|_\eps \leq \delta$ for some index $j\in \{1{:}N_\delta\}$.
Moreover, since $\bC^\infty_{\rm c}(\Dom)$ is dense in $\bL^2(\Dom)$, for each
$j\in \{1{:}N_\delta\}$, there exists $\widetilde \bv_j \in \bC^\infty_{\rm c}(\Dom)$
such that $\|\bv_j-\widetilde \bv_j\|_\eps \leq \delta$. We have therefore shown
that for all $\bw \in \calB$, there exists an index $j\in\{1{:}N_\delta\}$ such that
\begin{equation*}
\|\bw-\widetilde \bv_j\|_\eps \leq 2\delta.
\end{equation*}

We can now conclude by using the same arguments as in the above proof
(notice that $\{ \bw \in \bX\upc_0, \; \|\ROTZ\bw\|_\bnu \leq 1\} \subset \calB$).
\end{proof}

\bibliographystyle{plain}
\bibliography{biblio}

\end{document}